\documentclass[12pt]{amsart}
\usepackage{amsfonts,amsthm,amsmath,amssymb,verbatim}
\usepackage{graphicx}
\usepackage[active]{srcltx}
\usepackage[all,cmtip]{xy}

\newtheorem{thm}{Theorem}[section]
\newtheorem{prop}[thm]{Proposition}

\newtheorem{cor}[thm]{Corollary}

\theoremstyle{remark}

\newtheorem{example}[thm]{Example}
\newtheorem{remark}[thm]{Remark}
\newtheorem{defin}{Definition}



\def\C{\mathbb{C}}

\def\Z{\mathbb{Z}}
\def\P{\mathbb{P}}

\def\R{\mathbb{R}}

\def\cJ{\mathcal{J}}
\def\S{\mathfrak{S}}
\def\SP{\mathcal{SP}}
\def\GZ{\mathcal{GZ}}
\def\lb{\lambda}%\l-balanced

\def\start{{\rm start}}

\def\Oc{\mathcal{O}}
\def\a{\alpha}

\def\om{\omega}
\def\l{\lambda}
\def\G{\Gamma}

\def\w_0{\underline{w_0}}
\def\id{\rm id}

\def\emptyset{\varnothing}

\title{Geometric mitosis}
\author{Valentina Kiritchenko}
\email{vkiritch@hse.ru}
\thanks{The author was supported by AG Laboratory NRU HSE,
MESRF grants ag. 11.G34.31.0023 and MK-983.2013.1, and RScF grant 14-21-00053.
}

\address{Laboratory of Algebraic Geometry and Faculty of Mathematics\\
National Research University Higher School of Economics\\
Vavilova St. 7, 112312 Moscow, Russia}
\address{Institute for Information Transmission Problems, Moscow, Russia}

\date{}

\begin{document}
\maketitle
\begin{abstract}
We describe an elementary convex geometric algorithm for realizing Schubert cycles in complete flag varieties by unions of faces of polytopes.
For $GL_n$ and Gelfand--Zetlin polytopes, combinatorics of this algorithm coincides with that of  the  mitosis on pipe dreams introduced by Knutson and Miller.
For $Sp_4$ and a Newton--Okounkov polytope of the symplectic flag variety, the algorithm yields a new combinatorial rule that extends to $Sp_{2n}$.
\end{abstract}
\section{Introduction}
Positive presentations of Schubert cycles such as classical Schubert polynomials
play a key role in the Schubert calculus.
Ideas of toric geometry and theory of Newton (or moment) polytopes motivated search for positive presentations
with a more convex geometric flavor.
For instance, Schubert cycles on the complete flag variety for $GL_n$ were identified by various means with unions of faces of Gelfand--Zetlin polytopes \cite{KThesis,KoM,K10,KST}.
In the present paper, we develop an algorithm for representing Schubert cycles by faces of convex polytopes in the case of complete flag varieties for arbitrary reductive groups.

Let $G$ be a connected reductive group, $B\subset G$ a Borel subgroup, and $X=G/B$ the complete flag variety.
There are several partially overlapping classes of polytopes
that can be associated with ample line bundles on $X$ (see e.g. \cite{BZ,L,GK,Ka,K}).
For instance, string polytopes of Berenstein--Zelevinsky and Littelmann were recently exhibited in \cite{Ka} as Newton--Okounkov polytopes of flag varieties for a certain $B$-invariant valuation on $X$.
These polytopes usually form families $P_\l\subset\R^d$ (where $d:=\dim G/B$) parameterized by dominant weights $\l$ of $G$, and satisfy the property $|P_\l\cap\Z^d|=\dim V_\l$ where $V_\l$ is the irreducible $G$-module with the highest weight $\l$.
Moreover, there is a projection $p:\R^d\to\Lambda_G\otimes\R$ (where $\Lambda_G$ is the weight lattice of $G$) such that the Weyl character $\chi(V_\l)$ can be expressed as the multiplicity free sum over the lattice points in $P_\l$:
$$\chi(V_\l)=\sum_{x\in P_\l\cap\Z^d}e^{p(x)}.$$

Recall that ample line bundles $L_\l$ on $X$ are in bijective correspondence with irreducible representations $V_\l$ of $G$, 
and $H^0(X,L_\l)=V_\l^*$.
Let $X_w$ be a Schubert variety, i.e., the closure of a $B$-orbit in $G$, and  $\chi_w(\l):=\chi(H^0(X_w,L_\l|_{X_w})^*)$ the corresponding 
Demazure character. 
A natural way to identify $X_w$ with a union of faces $\S_w\subset P_\l$ is to choose $\S_w$ so that the following identity for holds for all $\l$:
$$\chi_w(\l)=\sum_{x\in \S_w\cap\Z^d}e^{p(x)}.$$
Demazure characters can be calculated inductively starting from the class of a point $X_{\id}=\{pt\}$ (that is, $\chi_{\id}(\l)=e^\l$), and applying Demazure operators $D_1$,\ldots,$D_r$ corresponding to the simple roots of $G$.
In this paper, we define geometric mitosis operations $M_1$,\ldots, $M_r$ on faces of $P_\l$ as convex geometric counterparts of Demazure operators $D_1$,\ldots, $D_r$, that is, they satisfy the identity
$$D_i\left(\sum_{x\in \S_w\cap\Z^d}e^{p(x)}\right)=\sum_{x\in M_i(\S_w)\cap\Z^d}e^{p(x)}.$$
whenever $l(s_iw)=l(w)+1$.

The definition of mitosis operations is elementary, and its main ingredient is mitosis on parallelepipeds introduced in \cite[Section 6]{KST}.
Mitosis on parallelepipeds can be viewed as a geometric realization of the mitosis  of Knutson--Miller \cite{KnM,M} restricted to two consecutive rows of pipe dreams.
We use mitosis on parallelepipeds as a building block for mitosis on more general polytopes $P_\l$ (called {\em parapolytopes}) that admit $r$ different fibrations by parallelepipeds.
Though the building blocks considered individually are combinatorially the same their arrangement depends significantly on the combinatorics of $P_\l$.
For instance, for Gelfand--Zetlin polytopes we get mitosis on usual pipe dreams, and for a polytope associated with the cone of adapted strings in type $C$ (see Section \ref{s.comb}) we get different combinatorial objects called {\em skew pipe dreams}.

This paper is organized as follows.
In Section \ref{s.para}, we recall mitosis on parallelepipeds and its relation to Demazure-type operators, and define geometric mitosis on parapolytopes.
In Section \ref{s.main}, we consider parapolytopes associated with reductive groups and prove Theorem \ref{t.Demazure} that relates Demazure operators with geometric mitosis.
In Corollary \ref{t.Schubert}, we give an algorithm for generating faces that represent a given Demazure character (or equivalently, a given Schubert cycle).
In Section \ref{s.Sp_4}, we apply the results of the preceding sections to $Sp_4$ and the symplectic DDO polytope $\SP_\l$ constructed in \cite{K}.
We prove that $\SP_\l$ can be realized as the Newton--Okounkov body of the flag variety $Sp_4/B$ and the line bundle $L_\l$ for 
a natural geometric valuation considered in \cite{An,Ka}.
Next, we outline how results of \cite{Ka2,KST} can be used to model the Schubert calculus on $Sp_4/B$ by intersecting faces of $\SP_\l$.
In Section \ref{s.comb}, we describe combinatorics of geometric mitosis, in particular, define mitosis on skew pipe dreams in type $C$ that generalize combinatorics of mitosis for $\SP_\l$.
We also formulate open questions.

I am grateful to Dave Anderson, Megumi Harada and Kiumars Kaveh for useful discussions.

\section{Mitosis on polytopes}\label{s.para}
\def\R{\mathbb{R}}
In this section, we define a convex-geometric operation ({\em geometric mitosis}) on faces of polytopes that models Demazure operators from representation theory.
The definition is elementary and reduces to the case of parallelepipeds, which we discuss first.
For special classes of polytopes associated with reductive groups, geometric mitosis has algebro-geometric and representation-theoretic meaning.
This will be discussed in the next section.

\subsection{Mitosis on parallelepipeds}
First, recall the {\em mitosis on parallelepipeds} (or {\em paramitosis}) from \cite[Section 6]{KST} using more geometric terms.
Let $\Pi:=\Pi(\mu,\nu)\subset\R^n$ be a parallelepiped given by inequalities
$\mu_i\le x_i\le \nu_i$ for $i=1$,\ldots, $n$.
In what follows, we will only consider parallelepipeds of this kind.
They will be called {\em coordinate parallelpipeds}.

\begin{defin} An edge of $\Pi$ is {\em essential} if it is given by equations
$$x_1=\mu_1,\ldots,x_{i-1}=\mu_{i-1}; \ \ x_{i+1}=\nu_{i+1},\ldots, x_{n}=\nu_n.$$
\end{defin}
Clearly, the number of essential edges is equal to $\dim\Pi$, and
the union of essential edges forms a broken line that connects
the vertices $(\mu_1,\ldots,\mu_n)$ and $(\nu_1,\ldots,\nu_n)$.
Denote the set of essential edges of $\Pi$ by $E(\Pi)$.

For every face $\Gamma\subset\Pi$, we now define a collection of faces $M(\Gamma)$.
Let $k$ be the minimal number such that $\Gamma\subseteq \{x_i=\mu_i\}$ for all $i>k$ (in particular, $\Gamma\nsubseteq \{x_k=\mu_k\}$) and $\nu_i\ne\mu_i$ for at least one $i>k$.
If no such $k$ exists then $M(\Gamma)=\emptyset$.
Under the isomorphism $\R^n\simeq\R^{k}\times\R^{n-k};$
$(x_1,\ldots,x_n)\mapsto (x_1,\ldots,x_k)\times(x_{k+1},\ldots,x_n)$ the parallelepiped $\Pi$
gets mapped to $\Pi'\times\Pi''$  where $\Pi'\subset\R^k$ and $\Pi''\subset\R^{n-k}$ are
coordinate parallelepipeds.
The face $\Gamma$ gets mapped to $\Gamma'\times v$ where $v=(\mu_{k+1},\ldots,\mu_{n})$
is a vertex of $\Pi''$ and $\Gamma'\subset\Pi'$ is
a face of $\Pi'$.

\begin{defin}\label{d.paramitosis}
The set $M(\Gamma)$ (called the {\em mitosis} of $\Gamma$)
consists of all faces $\Gamma'\times E$  such that $E$ is an essential
edge of $\Pi''$.
\end{defin}
In particular, $\dim\Delta=\dim \Gamma+1$ for any $\Delta\in M(\Gamma)$.
It is easy to check that $M^2(\Gamma)=\emptyset$ for any face $\Gamma$.
Here is the key example of mitosis.
\begin{example} If $\Gamma$ is the vertex $(\mu_{1},\ldots,\mu_{n})$, then $M(\Gamma)$ is the set of essential edges of $\Pi$.
\end{example}

This geometric version of mitosis is similar to the combinatorial mitosis of \cite{KnM}.
To see this represent every face of $\Pi(\mu,\nu)$ by a $2\times n$ table
$(a_{ij})_{i=1,2,\ 1\le j\le n}$ whose cells are either filled with $+$ or empty.
Namely, the face satisfies the equality $x_i=\mu_i$ or $x_i=\nu_i$ if and only if
$a_{1i}=+$ or $a_{2i}=+$, respectively (in particular, if $\mu_i=\nu_i$ then the $i$-th column
has two $+$).
On the level of tables, operation $M$ coincides the mitosis $M_i$ of \cite{KnM} on reduced
pipe dreams restricted to the rows $i$ and $i+1$ after reflecting our tables in a vertical line
(cf. \cite[Definition 6]{M}).

\begin{example}
If $\Pi(\mu,\nu)\subset\R^4$, where $\mu=(1,1,1,1)$ and $\nu=(2,2,1,2)$ (that is, $\mu_3=\nu_3$),
then the edge $\Gamma=\{\ x_2=\mu_2, \ x_4=\mu_4\}$ is represented by the table

$$\begin{array}{|c|c|c|c|}
\hline
\ \ &+ &+ &+\\
\hline
&&+&\\
\hline
\end{array}$$

The set $M(\Gamma)$ consists of two edges represented by the tables

$$\begin{array}{|c|c|c|c|}
\hline
\ \ &+ &+ &\ \ \\
\hline
&&+&\\
\hline
\end{array}\ \& \
\begin{array}{|c|c|c|c|}
\hline
\ \ &\ \  &+ & \\
\hline
&&+&+\\
\hline
\end{array}$$
\end{example}
The combinatorial notion of {\em chute moves} and {\em ladder moves} on
reduced RC-graphs or pipe dreams introduced in \cite{BB}
(cf. \cite[Definition 8]{M}) can also be extended to the geometric setting as follows.
For a partition $J=(0\le j_1<\ldots<j_{k-1}<j_k\le n)$, let $p_{J,i}$ denote the projection
$(x_1,\ldots,x_n)\mapsto (x_{j_i+1},\ldots,x_{j_{i+1}})$ for $i=1,\ldots,k-1$.
For $i=0$ and $i=k$ the projections $p_{J,i}$ are defined
by  $(x_1,\ldots,x_n)\mapsto (x_{1},\ldots,x_{j_1})$ and
$(x_1,\ldots,x_n)\mapsto (x_{{j_k}+1},\ldots,x_n)$, respectively.

\begin{defin}
A face $\Gamma\subset\Pi$ is called {\em reduced} if there is a partition $J(\Gamma)$ such that
$p_{J,1}(\Gamma)$ and $p_{J,k+1}(\Gamma)$ are the vertices $(\nu_1,\ldots,\nu_{j_1-1})$ and
$(\mu_{j_{k}+1},\ldots,\mu_n)$, respectively,
and $p_{J,i}(\Gamma)\in E(p_{J,i}(\Pi))$ for any $i=2$,\ldots, $k$.
\end{defin}
A partition $J(\Gamma)$ is unique if we ignore all indices $i\in\{1,\ldots,n\}$
such that $\mu_i=\nu_i$.
In the language of \cite{M}, the partition $J(\Gamma)$ corresponds to the decomposition of pipe dreams into {\em introns}.
The length of  $J(\Gamma)$ is equal to $\dim\G+2$.

\begin{defin}
Two reduced faces $\Gamma$ and $\Gamma'$ are said to be {\em $L$-equivalent} if
$J(\Gamma)=J(\Gamma')$.
Denote by $L(\Gamma)$ the set of all reduced faces equivalent to $\Gamma.$
\end{defin}

\begin{example}
Take $\Pi(\mu,\nu)\subset\R^5$, where $\mu=(1,1,1,1,1)$ and $\nu=(2,2,1,2,2)$.
The face $\Gamma=\{x_1=\mu_1; \ x_4=\nu_4 \}$ of dimension $2$ is reduced with respect to the
partition $(0,4,5)$.
The set $L(\Gamma)$ consists of three faces represented by the tables
$$
\begin{array}{|c|c|c|c|c|c|}
\hline
+&+ &+& &\ \ \\
\hline
 &  &+&\ \ &\\
\hline
\end{array} \quad
\begin{array}{|c|c|c|c|c|c|}
\hline
+&\ \ &+& &\ \ \\
\hline
 &    &+&+&\\
\hline
\end{array}\ (=\Gamma) \quad
\begin{array}{|c|c|c|c|c|c|}
\hline
\ \ &\ \ &+& &\ \ \\
\hline
 &  +    &+&+&\\
\hline
\end{array} \ .
$$
\end{example}

\begin{remark}\label{r.LM}
There is a bijection between $L$-equivalence classes and faces of the standard simplex (see \cite[Proposition 6.6]{KST}),
which yields a minimal realization of the simplex as a cubic complex.
Using this bijection it is not hard to check that for any face $\G\subset\Pi$
the mitosis applied to faces in $L(\G)$ produces a single $L$--equivalence class, i.e.,
$$ \bigcup_{F\in L(\Gamma)\atop{E\in M(F)}}E=L(\G')$$
for any $\G'\in M(\G)$ (see \cite[Remark 6.7]{KST} for more details).
\end{remark}

Definitions of $M(\Gamma)$ and $L(\Gamma)$ are motivated by the identity
\cite[Proposition 6.8]{KST} for a Demazure-type operator applied to an exponential sum over $\Gamma$.
We briefly recall this identity (for more details see \cite[Section 6]{KST}).
Let $s:\Z\to\Z$ be a reflection about
$$C:=\frac12\sum_{i=1}^n(\mu_i+\nu_i),$$
that is, $s(k)=2C-k$ for $k\in\Z$.
The reflection $s$ acts on the Laurent polynomials $\Z[t,t^{-1}]$ by $s(t^k)=t^{s(k)}$.
Define the operator $T_\Pi$ on $\Z[t,t^{-1}]$ by the formula
$$
T_\Pi f=\frac{f-t\cdot s(f)}{1-t}.
$$
It is not hard to see that, for every Laurent polynomial $f$, the function
$T_\Pi f$ is also a Laurent polynomial.
The operator $T_\Pi$ depends on the parallelepiped $\Pi=\Pi(\mu,\nu)$.

For a subset $A\subset\Pi(\mu,\nu)$, we define the Laurent polynomial
$\chi(A):=\sum_{x\in A\cap\Z^n} t^{\sigma(x)}\in\Z[t,t^{-1}]$ where
$\sigma(x):=\sum_{i=1}^n x_i$.

\begin{prop}{\cite[Proposition 6.8]{KST}}
\label{p.paramitosis}
Let $\Gamma$ be a reduced face of $\Pi$ such that $\Gamma$ contains the vertex $(\mu_1,\ldots,\mu_n)$.
Then $$T_\Pi\chi\left(\bigcup_{F\in L(\Gamma)}F\right)=
\chi\left(\bigcup_{F\in L(\Gamma)\atop{E\in M(F)}}E\right).$$
\end{prop}

\begin{example} The simplest but crucial example is $\Gamma=\{(\mu_1,\ldots,\mu_n)\}$.
Then $L(\Gamma)=\{\Gamma\}$ and $M(\Gamma)=E(\Pi)$.
Hence, $\chi(\Gamma)=t^{\sigma(\mu)}$ and
$\chi(\bigcup_{E\in M(\Gamma)}E)=\sum_{i=\sigma(\mu)}^{\sigma(\nu)}t^i$.
The above proposition reduces to the geometric progression sum formula:
$$\frac{t^{\sigma(\mu)}-t\cdot t^{\sigma(\nu)}}{1-t}=
\sum_{i=\sigma(\mu)}^{\sigma(\nu)}t^i.$$
It is not hard to deduce Proposition \ref{p.paramitosis} from this partial case.
\end{example}
In what follows, we sometimes denote mitosis on parallelepipeds by $M_\Pi$ to indicate which parallelepiped $\Pi$ to consider.

\subsection{Mitosis on parapolytopes}
\def\Z{\mathbb{Z}}
\def\a{\alpha}
\def\l{\lambda}
We now use mitosis on parallelpipeds to define mitosis on a more general class of polytopes, namely,
on {\em parapolytopes}.
Consider the space with the direct sum decomposition
$$\R^d=\R^{d_1}\oplus\ldots\oplus\R^{d_r}$$
and choose coordinates $x=(x_1^1,\ldots,x_{d_1}^1;\ldots;x_1^r,\ldots,x_{d_r}^r)$ with respect
to this decomposition.

\begin{defin}
A convex polytope $P\subset{\mathbb R}^d$ is called a {\em parapolytope} if for any $i=1$,\ldots, $r$,
and any vector $c\in\R^d$ the intersection of $P$
with the parallel translate of $\R^{d_i}$ by $c$
is either empty or the parallel translate of a coordinate parallelepiped in $\R^{d_i}$, i.e.,
$$P\cap(c+\R^{d_i})=c+\Pi(\mu_c,\nu_c)$$
for $\mu_c$ and $\nu_c$ that depend on $c$.
\end{defin}

\begin{example} \label{e.GZ}
Consider the decomposition $\R^d=\R^{n-1}\oplus\R^{n-2}\oplus\ldots\oplus\R$ (that is,
$r=n-1$ and $d=\frac{n(n-1)}{2}$).
Let $\lambda=(\lambda_1,\ldots,\lambda_n)$ be a non-decreasing collection of real numbers.
For every $\l$, define the {\em Gelfand--Zetlin polytope} $\GZ_\l$ by the inequalities
$$
\begin{array}{cccccccccc}
\l_1&       & \l_2    &         &\l_3          &    &\ldots    & &       &\l_n   \\
    &x^1_1&         &x^1_2  &         & \ldots   &       &  &x^1_{n-1}&       \\
    &       &x^2_1 &       &  \ldots &   &        &x^2_{n-2}&         &       \\
    &       &       &  \ddots   & &  \ddots   &      &         &         &       \\
    &       &       &  &x^{n-2}_1&     &  x^{n-2}_2 &        &         &       \\
    &       &         &    &     &x^{n-1}_1&   &              &         &       \\
\end{array}
$$
where the notation
$$
 \begin{array}{ccc}
  a &  &b \\
   & c &
 \end{array}
 $$
means $a\le c\le b$.
It is easy to check that $\GZ_\l$ is a parapolytope.
\end{example}

If $P\subset\R^d$ is a parapolytope then we can define
$r$ different mitosis operations $M_1$,\ldots, $M_r$ on faces of $P$.
These operations come from mitosis on parallelepipeds $P_\l\cap(c+\R^{d_1})$,\ldots, $P_\l\cap(c+\R^{d_r})$, respectively.
For a polytope $\Gamma\subset \R^d$, denote by $\Gamma^\circ$ the relative interior of $\Gamma$,
i.e., $\Gamma^\circ$ consists of all points of $\Gamma$ that do not lie in faces of smaller dimension.

\begin{defin}\label{d.mitosis} Let $i=1$,\ldots, $r$, and $\Gamma$ a face of $P$.
Choose $c\in \Gamma^\circ$.
Put $\Pi_c:=P\cap(c+\R^{d_i})$ and $\Gamma_c:=\Gamma\cap(c+\R^{d_i})$.
The set $M_i(\Gamma)$ consists of all faces $\Delta\subset P$ such that $\Delta^\circ$ contains $F^\circ$ for some $F\in M_{\Pi_c}(\Gamma_c)$.
Here  $M_{\Pi_c}$ is the mitosis on the parallelepiped $\Pi_c$ (see Definition \ref{d.paramitosis}).
\end{defin}
It is easy to check that $M_i(\Gamma)$ does not depend on the choice of $c\in\Gamma^\circ$.
Similarly, we can define the {\em L-class} $L_i(\Gamma)$ if $\Gamma_c$ is reduced.
\begin{defin}\label{d.ladder} Let $i=1$,\ldots, $r$, and $\Gamma$ a face of $P$.
We say that $\Gamma$ is {\em $L_i$-reduced} if $\Gamma_c:=\Gamma\cap(c+\R^{d_i})$ is reduced for some
$c\in \Gamma^\circ$. 
\end{defin}

\begin{example} Consider Example \ref{e.GZ} for $n=3$.
There will be two mitosis operations $M_1$, $M_2$.
Let us apply compositions of $M_1$ and $M_2$ to the vertex $a_\l=\{x^1_1=x^2_1=\l_1;\ x^1_2=\l_2\}$ (i.e., the vertex with the lowest sum of coordinates).
The resulting faces will all contain $a_\lambda$, and hence, can be encoded by the following table:
$$
\begin{array}{|c|c|}
\hline
+\Leftrightarrow x^1_1=\l_1&+ \Leftrightarrow x^1_2=\l_2\\
\hline
 & +\Leftrightarrow x^2_1=\l_1\\
\hline
\end{array}\ , \mbox{ e.g. the face } \{x^1_1=\l_1\} \mbox{ is encoded by }
\begin{array}{|c|c|}
\hline
+& \ \ \\
\hline
 &\\
\hline
\end{array}\ .
$$
Applying Definition \ref{d.mitosis} repeatedly, we get
$$
a_\l=\begin{array}{|c|c|}
\hline
+ &+ \\
\hline
 & +\\
\hline
\end{array}
\stackrel{M_1}{\longrightarrow}
\begin{array}{|c|c|}
\hline
 \ \ &+ \\
\hline
 & +\\
\hline
\end{array}
\stackrel{M_2}{\longrightarrow}
\begin{array}{|c|c|}
\hline
\ \  &+ \\
\hline
 & \\
\hline
\end{array}
\stackrel{M_1}{\longrightarrow}
\begin{array}{|c|c|}
\hline
& \ \ \\
\hline
\ \ & \\
\hline
\end{array}=\GZ_\l
$$
$$
a_\l
\stackrel{M_2}{\longrightarrow}
\begin{array}{|c|c|}
\hline
+ &+ \\
\hline
 & \\
\hline
\end{array}
\stackrel{M_1}{\longrightarrow}
\left\{\begin{array}{|c|c|}
\hline
+  &\ \ \\
\hline
 & \\
\hline
\end{array}\ , \
\begin{array}{|c|c|}
\hline
\ \ &  \\
\hline
& +\\
\hline
\end{array}\right\}
\stackrel{M_2}{\longrightarrow}
\GZ_\l
$$
From a combinatorial viewpoint, this is exactly mitosis on pipe dreams of \cite{KnM} (after reflecting our diagrams
in a vertical line).
For arbitrary $n$, geometric mitosis on $\GZ_\l$ also yields combinatorial mitosis on pipe dreams
(see \cite[Section 6.3]{KST}).
\end{example}
We now consider an example where geometric mitosis produces a new combinatorial rule.
\begin{example} \label{e.Sp_4}
Let $\l=(\l_1,\l_2)$, where $\l_1$ and $\l_2$ are positive real numbers.
In \cite[Example 3.4]{K}, convex-geometric divided difference operators were used to construct the following symplectic DDO polytope $\SP_\lambda$ in $\R^4$:
$$0\le y_1\le \l_1, \quad y_2\le y_1+\l_2, \quad y_3\le 2y_2,$$
$$y_3\le y_2+\l_2,\quad 0\le y_4\le \l_2, \quad y_4\le \frac{y_3}{2}.$$
As can be readily seen from the inequalities, it is a parapolytope with respect to the decomposition $\R^4=\R^2\oplus\R^2$
given by $x_1^1=y_1$, $x^2_1=y_2$, $x^1_2=y_3$, $x^2_2=y_4$.
Hence, there are two mitosis operations $M_1$ and $M_2$.
Again, let us apply compositions of $M_1$ and $M_2$ to the lowest (with respect to the sum of coordinates) vertex $0\in \SP_\l$.
The faces of $\SP_\l$ that contain $0$ can be encoded by the following diagram:
$$\begin{array}{|c|}
\hline
+ \Longleftrightarrow 0=y_1\\
\hline
\end{array}\begin{array}{|c|}
\hline
+ \Longleftrightarrow 0=y_4 \\
\hline
+ \Longleftrightarrow y_4=\frac{y_3}{2}\\
\hline
+ \Longleftrightarrow y_3=2y_2\\
\hline
\end{array}\ , \mbox{ e.g. } \{y_1=0, \ y_3=2y_2\} \mbox{ is encoded by }
\begin{array}{|c|}
\hline
+ \\
\hline
\end{array}\begin{array}{|c|}
\hline
\\
\hline
\\
\hline
+\\
\hline
\end{array}.
$$
By Definition \ref{d.mitosis} we get
$$0=
\begin{array}{|c|}
\hline
+\\
\hline
\end{array}\begin{array}{|c|}
\hline
+\\
\hline
+ \\
\hline
+ \\
\hline
\end{array}
\stackrel{M_1}{\longrightarrow}
\begin{array}{|c|}
\hline
\ \ \\
\hline
\end{array}\begin{array}{|c|}
\hline
+\\
\hline
+ \\
\hline
+ \\
\hline
\end{array}
\stackrel{M_2}{\longrightarrow}
\begin{array}{|c|}
\hline
\ \ \\
\hline
\end{array}\begin{array}{|c|}
\hline
+\\
\hline
+ \\
\hline
 \\
\hline
\end{array}
\stackrel{M_1}{\longrightarrow}
\begin{array}{|c|}
\hline
\ \ \\
\hline
\end{array}\begin{array}{|c|}
\hline
+\\
\hline
 \\
\hline
 \\
\hline
\end{array}
\stackrel{M_2}{\longrightarrow}
\begin{array}{|c|}
\hline
\ \ \\
\hline
\end{array}\begin{array}{|c|}
\hline
\\
\hline
\ \  \\
\hline
\\
\hline
\end{array} =\SP_\l$$
$$0
\stackrel{M_2}{\longrightarrow}
\begin{array}{|c|}
\hline
+\\
\hline
\end{array}\begin{array}{|c|}
\hline
+\\
\hline
+ \\
\hline
 \\
\hline
\end{array}
\stackrel{M_1}{\longrightarrow}
\left\{\begin{array}{|c|}
\hline
\ \ \\
\hline
\end{array}\begin{array}{|c|}
\hline
+\\
\hline
 \\
\hline
+ \\
\hline
\end{array}\ , \
\begin{array}{|c|}
\hline
+ \\
\hline
\end{array}\begin{array}{|c|}
\hline
+\\
\hline
\\
\hline
 \\
\hline
\end{array}
\right\}
\stackrel{M_2}{\longrightarrow}
\left\{\begin{array}{|c|}
\hline
\ \ \\
\hline
\end{array}\begin{array}{|c|}
\hline
\\
\hline
+ \\
\hline
 \\
\hline
\end{array} \ , \
\begin{array}{|c|}
\hline
\ \ \\
\hline
\end{array}\begin{array}{|c|}
\hline
\\
\hline
 \\
\hline
+ \\
\hline
\end{array}
\ , \
\begin{array}{|c|}
\hline
+\\
\hline
\end{array}\begin{array}{|c|}
\hline
\\
\hline
\ \  \\
\hline
\\
\hline
\end{array}\right\}
\stackrel{M_1}{\longrightarrow}
\SP_\l
$$
\end{example}
The combinatorics of the last example can be extended to the decomposition $\R^{r^2}=\R^{r}\oplus\R^{2r-2}\oplus\R^{2r-4}\oplus\ldots\oplus\R^2$ (see Section \ref{s.comb}).

\section{Geometric mitosis and Demazure operators}\label{s.main}
In this section, we discuss the relation between geometric mitosis, Demazure operators and Schubert calculus.
We introduce a  special class of parapolytopes associated with reductive groups.
In particular, Gelfand--Zetlin polytopes and, more generally, polytopes constructed in \cite[Section 3]{K} via convex-geometric divided difference operators belong to this class.

Let $G$ be a connected reductive group of semisimple rank $r$.
Let $\alpha_1$,\ldots, $\alpha_r$ denote simple roots of $G$, and
$s_1$,\ldots, $s_r$ the corresponding simple reflections.
Fix a reduced decomposition $\w_0=s_{i_1}s_{i_2}\cdots s_{i_d}$ of the longest element $w_0$
of the Weyl group of $G$.
Let $d_i$ be the number of $s_{i_j}$ in this decomposition such that $i_j=i$.
Consider the space
$$\R^d=\R^{d_1}\oplus\ldots\oplus\R^{d_r}.$$
As before, we choose coordinates $x=(x_1^1,\ldots,x_{d_1}^1;\ldots;x_1^r,\ldots,x_{d_r}^r)$ with respect to this decomposition.
We will also use an alternative labeling of coordinates $(y_1,\ldots,y_d)$ where $$y_{d-j+1}=x^{i_j}_{p_j}$$ for $p_j:= \{k\ge j \ |\ s_{i_k}=s_{i_j}\}$.
\begin{example}\label{e.GZ1}
(a) Let $G=GL_n$ and $\w_0=(s_1)(s_2s_1)(s_3s_2s_1)\ldots(s_{n-1}\ldots s_1)$.
Then $r=n-1$, $d=\frac{n(n-1)}{2}$ and $\R^d=\R^{n-1}\oplus\R^{n-2}\oplus\ldots\oplus\R$.
The labelings of coordinates are related as follows:
$$(y_1,y_2,\ldots,y_d)=(x^1_1,x^2_1,\ldots,x^{n-1}_1;x^1_2, x^2_2,\ldots, x^{n-2}_2;\ldots;x_1^{n-1}).$$

(b) Let $G=Sp_4$ and $\w_0=s_2s_1s_2s_1$ (the symplectic DDO polytope $\SP_\l$ was constructed in \cite[Example 3.4]{K} using this decomposition).
Then $r=2$, $d=4$, $\R^4=\R^2\oplus\R^2$, and
$$(y_1,y_2,y_3,y_4)=(x^1_1,x^2_1,x^1_2,x^2_2)$$
exactly as in Example \ref{e.Sp_4}.
\end{example}

Put $\sigma_i(x)=\sum_{j=1}^{d_i} x^i_j$.
Let $\Lambda_G$ denote the weight lattice of $G$.
Define the projection $p$ of $\R^d$ to $\Lambda_G\otimes\R$ by the
formula $p(x)=\sigma_1(x)\alpha_1+\ldots+\sigma_r(x)\alpha_r$.
In what follows, we always assume that $P$ lies in the positive octant and contains the origin, that is, the origin is the vertex of $P$ with the minimal sum of coordinates.
Let  $\lambda$ be a dominant weight of $G$.
In what follows, we identify $\R^{d}/\R^{d_i}$ with $\R^{d_1}\oplus\ldots\oplus\widehat{\R^{d_i}}\oplus\ldots\oplus\R^{d_r}$.
\begin{defin}
Let $i\in\{1,\ldots,r\}$.
A parapolytope $P\subset \R^d$ is called {\em  $(\lambda, i)$-balanced}
if for any $c\in\R^d/\R^{d_i}$ we have
$$\sigma_i(\mu_c)+\sigma_i(\nu_c)=(-w_0\l-p(c),\a_i),$$
where $(\cdot,\a_i)$ is a coroot, i.e., is defined by the identity $s_i(\chi)=\chi-(\chi,\a_i)\a_i$ for all $\chi$ in the weight lattice.
\end{defin}

\begin{example}\label{e.GZ_Sp}
We continue Example \ref{e.GZ1}.

(a) Let $a_\l:=(\l_1,\ldots,\l_{n-1}$; $\l_1,\ldots,\l_{n-2};\ldots; \l_1)$ be the lowest vertex of the Gelfand--Zetlin polytope $\GZ_\l$ (see Example \ref{e.GZ}).
Let $\om_1$,\ldots, $\om_{n-1}$ denote the fundamental weights of $SL_n$.
It is easy to check that the parallel translate $\GZ_\l-a_\l$ of the Gelfand--Zetlin polytope
is $(\l,i)$-balanced for all $i\in\{1,\ldots,n-1\}$ and $\l=(\l_2-\l_1)\om_1+\ldots+(\l_n-\l_{n-1})\om_1$.

(b) Let $\l$ be a strictly dominant weight of $Sp_4$.
Let $\a_1$ denote the shorter root, and $\a_2$ the longer one.
Put $\l_i=(\l,\a_i)$ for $i=1,2$.
It is easy to check that the symplectic DDO polytope $\SP_\l$ from Example \ref{e.Sp_4} is $(\l,i)$-balanced for $i=1,2$.
\end{example}

\begin{defin} A parapolytope $P\subset\R^d$ is called $\lambda$-{\em balanced} if
it is $(\l,i)$-balanced for any $i\in\{1,\ldots,r\}$
\end{defin}
In particular, the polytopes considered in Examples \ref{e.GZ_Sp} are $\lambda$-balanced.
For certain $\w_0$, one can construct $\l$-balanced polytopes using an elementary convex-geometric algorithm that mimics divided difference operators (see \cite[Theorem 3.6]{K} for more details), e.g. Gelfand--Zetlin polytopes and the symplectic DDO polytope $\SP_\l$ can be constructed this way.
Another source of $\l$-balanced polytopes might be provided by Newton--Okounkov polytopes of flag varieties for certain valuations.
For instance,  $\SP_\l$ can also be realized as the Newton--Okounkov polytope of the flag variety of $Sp_4$ 
for a geometric valuation associated with $\w_0$ (see Section \ref{s.Sp_4}).

\begin{remark} \label{r.DDO}
The symplectic DDO polytope $\SP_\l$ has 11 vertices, hence, it is not combinatorially equivalent to string polytopes for $Sp_4$ and $\w_0=s_1s_2s_1s_2$ or $s_2s_1s_2s_1$ defined in \cite{L} (the latter have 12 vertices).
\end{remark}

If $P_\l$ is a $\lb$-balanced parapolytope, then geometric mitosis on $P_\l$ is compatible
with the action of Demazure operators $D_{\a_1}$,\ldots, $D_{\a_r}$ on the group algebra $\Z[\Lambda_G]$.
Let $\a$ be a root of $G$.
Recall that $D_\a$ acts on $\Z[\Lambda_G]$ as follows:
$$D_\a e^{\mu}=\frac{e^{\mu}-e^{\a}e^{s_i(\mu)}}{1-e^{\a}}.$$
For a subset $A\subset P_\l$, denote by $A_c$ the intersection $A\cap(c+\R^{d_i})$.
Let $\pi_i:\R^d\to\R^{d_1}\oplus\ldots\oplus\widehat{\R^{d_i}}\oplus\ldots\oplus\R^{d_r}$ be the projection that forgets coordinates $(x^i_1,\ldots,x^i_{d_i})$.
\begin{thm} \label{t.Demazure}
Let $i\in\{1,\ldots, r\}$, and $S$ a collection of $L_i$-reduced faces of a $\lb$-balanced parapolytope $P_\l$ that satisfy the following conditions.

(1) Every $F\in S$ contains the vertex $0\in P_\l$.

(2) If $F\in S$, then $L_i(F)\subset S$.

(3) For every $F\in S$ with empty $M_i(F)$ there exists $F'\in S$ with nonempty $M_i(F')$ such that $F_c\subset \G_c$ for some $\G\in M_i(F')$ and some $c\in F^\circ$.

(4) The sets $\S:=\bigcup_{F\in S}F$ and $M_i(\S):=\bigcup_{F\in S}\bigcup_{E\in M_i(F)}E$ have the same image under $\pi_i$, i.e., $\pi_i(\S)=\pi_i(M_i(\S))$.

Then we have
$$D_{\a_i}\left(e^{w_0\l}\sum_{x\in\S\cap\Z^d}e^{p(x)}\right)=
e^{w_0\l}\sum_{M_i(\S)\cap\Z^d}e^{p(x)}.$$

\end{thm}
\begin{proof}
Every $x\in P_\l$ can be written uniquely as $\pi_i(x)+z$ where $z\in\Pi_c$. 
Since $p(x)=p(\pi_i(x))+\sigma_i(z)\a_i$ we get
$$\sum_{x\in\S\cap\Z^d}e^{p(x)}=\sum_{c\in \pi_i(\S)\cap\Z^{d-d_i}}e^{p(c)}\sum_{z\in\S_c\cap\Z^{d_i}}t^{\sigma_i(z)},$$
where $t:=e^{\a_i}$.
Note that $D_{\a_i}(e^{p(c)+w_0\l}t^{\sigma_i(z)})=e^{p(c)+w_0\l}T_{\Pi_c}(t^{\sigma_i(z)})$ because $P_\l$ is $\lb$-balanced.
Hence,
$$D_{\a_i}\left(e^{w_0\l}\sum_{c\in \pi_i(\S)\cap\Z^{d-d_i}}e^{p(c)}\sum_{z\in\S_c\cap\Z^{d_i}}t^{\sigma_i(z)}\right)=
e^{w_0\l}\sum_{c\in \pi_i(\S)\cap\Z^{d-d_i}}e^{p(c)}T_{\Pi_c}\left(\sum_{z\in\S_c\cap\Z^{d_i}}t^{\sigma_i(z)}\right)$$
where $T_{\Pi_c}$ is the operator defined in Section \ref{s.para}.
By \cite[Proposition 6.10]{KST}, which is applicable because of hypotheses (1)--(3), we get
$$T_{\Pi_c}\left(\sum_{z\in\S_c\cap\Z^{d_i}}t^{\sigma_i(z)}\right)=
\sum_{z\in M_i(\S)_c\cap\Z^{d_i}}t^{\sigma_i(z)}.$$
Hence,
$$\sum_{c\in \pi_i(\S)\cap\Z^{d-d_i}}e^{p(c)}T_{\Pi_c}\left(\sum_{z\in\S_c\cap\Z^{d_i}}t^{\sigma_i(z)}\right)=
\sum_{c\in \pi_i(\S)\cap\Z^{d-d_i}}e^{p(c)}\sum_{z\in M_i(\S)_c\cap\Z^{d_i}}t^{\sigma_i(z)}$$
Finally, since $\pi_i(\S)=\pi_i(M_i(\S))$ by (4) we get
$$\sum_{c\in \pi_i(\S)\cap\Z^{d-d_i}}e^{p(c)}\sum_{z\in M_i(\S)_c\cap\Z^{d_i}}t^{\sigma_i(z)}=\sum_{M_i(\S)\cap\Z^{d}}e^{p(x)}.$$
\end{proof}
This theorem gives an inductive algorithm for realizing every Demazure character as the exponential sum over the unions of certain faces of $P_\l$ if $P_\l$ satisfies an extra assumption.
\begin{defin} A $\l$-balanced parapolytope $P_\l\subset\R^d$ with the lowest vertex $0$ is called
{\em admissible} if $\dim P\cap\R^{d_i}\le 1$ for all $i=1,\ldots,r$.
\end{defin}
\begin{remark} DDO polytopes of \cite[Section 3]{K} are admissible (see the discussion at the end of \cite[Section 4.3]{K}).
In particular, polytopes $\GZ_\l-a_\l$ and $\SP_\l$ are admissible, which is easy to check directly.
\end{remark}
We now discuss the algorithm.
Let $B\subset G$ be a Borel subgroup, and $X=G/B$ complete flag variety.
For an element $w\in W$ of the Weyl group, denote by $X_w=\overline {BwB/B}$ the Schubert variety corresponding to $w$.
We will also consider the opposite Schubert varieties $X^w=\overline{B^-wB/B}$ where  $B^-\subset G$ denotes the opposite Borel subgroup.
Note that Schubert cycles $[X^{w_0w}]$ and $[X_w]$ coincide in $H^*(G/B,\Z)$.
Recall that with a dominant weight $\l$ of $G$, one can associate a $G$-linear line bundle $L_\l$ on the complete flag variety $X=G/B$ so that $H^0(X,L_\l)=V_\l^*$ as $G$-modules.

The {\em Demazure $B$-submodule} $V^+_{\l,w}$ can be defined as $H^0(X_w, L_\l|_{X_w})^*$.
Similarly, {\em Demazure $B^-$-submodule} $V^-_{\l,w}$ can be defined as $H^0(X_w, L_\l|_{X^w})^*$.
Let $\chi_w(\l)$ and $\chi^w(\l)$ denote the characters of $V^+_{\l,w}$  and $V^-_{\l,w}$, respectively (they are called
{\em Demazure characters}).
It is easy to check that $w_0\chi_w(\l)=\chi^{w_0w}(\l)$.
Let $s_{j_1}\ldots s_{j_\ell}$ be a reduced decomposition of $w_0ww_0^{-1}$ such that
$(j_1,\ldots,j_\ell)$ is a subword of $(i_1,\ldots,i_d)$.
\begin{cor}\label{t.Schubert}
Let $P_\l\subset \R^{d}$ be an admissible $\lb$-balanced parapolytope,
and $\S_w\subset P_\l$ the union of all faces produced from the vertex $0\in P_\l$ by applying
successively the operations $M_{j_\ell}$,\ldots, $M_{j_1}$.
Suppose that for every $1<k\le\ell$, the collection of faces $M_{j_k}\ldots M_{j_\ell}(0)$
satisfies conditions (3) and (4) of Theorem \ref{t.Demazure}.
Then
$$\chi^{w_0w}(\l)=e^{w_0\l}\sum_{x\in \S_w\cap\Z^d}e^{p(x)}.$$
\end{cor}
\begin{proof} By the Demazure character formula \cite{A} we have
$$\chi^{w_0w}(\l)=D_{\a_{j_1}}\ldots D_{\a_{j_\ell}}e^{w_0\l}.$$
We now proceed by induction applying Theorem \ref{t.Demazure} repeatedly to the right hand side.
Note that conditions (1) and (2) of this theorem are fulfilled for $M_{i_k}\ldots M_{i_1}(0)$ for all $k<\ell$.
Indeed, if a face $\G$ contains $0$ then all faces in $M_i(\G)$ contain $0$ because $P_\l$ is admissible, and by Remark \ref{r.LM} the mitosis applied to a union of $L$-classes produces a union of $L$-classes.
\end{proof}
For $G=GL_n$ and $\GZ_\l-a_\l$, this corollary reduces to \cite[Theorem 5.1]{KST} and holds for all $w\in W$.
It is easy to check that for $G=Sp_4$ and $\SP_\l$, conditions of Corollary \ref{t.Schubert} are also satisfied for all $w$.
More generally, condition (4) is satisfied for all $w$ if $P_\l$ is a DDO polytope of \cite[Theorem 3.6]{K} (simply by construction of these polytopes).
Condition (3) is trickier to check as the case of Gelfand-Zetlin polytopes shows (see \cite[Lemma 6.13]{KST}).
Whenever Corollary \ref{t.Schubert} holds for all $w\in W$, the general results of \cite[Section 2]{KST} on polytope rings allow one to model Schubert calculus on $G/B$ by intersecting faces of $P_\l$.
For $GL_n$ and Gelfand--Zetlin polytopes this was done in \cite{KST}, and the example
with $Sp_4$ and $\SP_\l$ will be considered in the next section.

\section{$Sp_4$ example}\label{s.Sp_4}
We now apply the results of the preceding section to $Sp_4$ and the symplectic DDO polytope $\SP_\l$ from Example \ref{e.Sp_4}.
We explain an algebro-geometric meaning of $\SP_\l$ and outline applications of Corollary \ref{t.Schubert} to the Schubert calculus on $Sp_4/B$.
\subsection{DDO polytope as Newton--Okounkov body}
First, let us discuss the algebro-geometric interpretation of the symplectic DDO polytope.
Recall that $\a_1$ denotes the shorter root and $\a_2$ denotes the longer one.
Let $\om_1$, $\om_2$ be the corresponding fundamental weights, and
$\l=\l_1\om_1+\l_2\om_2$ a dominant weight of $Sp_4$.
We are going to identify $\SP_\l$ with the Newton-Okounkov polytope of $L_\l$ for a natural geometric valuation $v$ on $X$.

To define the valuation $v$ we introduce coordinates on an open Schubert cell in $X$.
Choose a basis  in $\C^4$ so that $\omega:=e_1^*\wedge e_4^*+e_2^*\wedge e_3^*$ is the symplectic form preserved by $Sp_4$.
Points in $X$ can be identified with {\em isotropic} complete flags $(V^1\subset V^2\subset V^3\subset\C^4)$. 
A flag is {\em isotropic} if $\omega|_{V^2}=0$ and $V^3={V^1}^\bot:=\{v\in\C^4\ | \ \omega(v,u)=0 \ \forall u\in V^1\}$.
Taking projectivization we also identify points in $X$ with projective partial flags $(a=\P(V^1)\in l=\P(V^2))$.
Fix the flag $(a_0,l_0)\in X$  where $a_0=(1:0:0:0)$ and $l_0=\langle a_0, (0:1:0:0)\rangle$, i.e., $(a_0,l_0)$ is the fixed point for the upper-triangular Borel subgroup $B\subset Sp_4$.
The open Schubert cell in $X$ with respect $(a_0,l_0)$ consists of all $(a,l)$ such that $(a_0,l_0)$ and $(a,l)$ are in general position (i.e., $a_0\notin l$, $a\notin l_0$, $l_0\cap l=\emptyset$ etc).
The Schubert varieties with respect to $(a_0,l_0)$ can be described as follows:
$$X_{\id}=\{(a_0,l_0)\}; \quad X_{s_1}=\{l=l_0\};\quad X_{s_2}=\{a=a_0\};$$
$$\quad X_{s_1s_2}=\{a\in l_0\}; \quad X_{s_2s_1}=\{a_0\in l\}; \quad X_{s_1s_2s_1}=\{l\cap l_0\ne\emptyset\};$$
$$X_{s_2s_1s_2}=\{a\in a_0^\bot\};  \quad X_{s_1s_2s_1s_2}=X_{s_2s_1s_2s_1}=X.
$$

Define coordinates on the open Schubert cell:
$$a=(y+xz:z:-x:1); \quad l=\langle a, (z+xt:t:1:0)\rangle.$$
These coordinates are  chosen so that the flag $\{x=y=z=t=0\}\subset\{x=y=z=0\}\subset\{x=y=0\}\subset\{x=0\}\subset X$
coincides with the flag of translated Schubert subvarieties:
$$s_1s_2s_1s_2X_{id}\subset s_1s_2s_1X_{s_2}\subset s_1s_2X_{s_1s_2}\subset s_1X_{s_2s_1s_2}\subset X$$
(after intersecting with the open Schubert cell).
The flag corresponds to the decomposition $\w_0=s_1s_2s_1s_2$, and the coordinates $(x,y,z,t)$ come naturally if one considers the Bott--Samelson variety $\widetilde X_{\w_0}$ (see \cite[Section 2.2]{Ka}).
Fix the lexicographic ordering on monomials in $x$, $y$, $z$, $t$, i.e. , $x^{k_1}y^{k_2}z^{k_3}t^{k_4}\succ x^{l_1}y^{l_2}z^{l_3}t^{l_4}$ iff
there exists $j\le4$ such that $k_i=l_i$ for $i<j$ and $k_j>l_j$.
Let $v:=v^{\w_0}$ denote the lowest order term valuation on $\C(X)$ associated with the flag and ordering  (cf. \cite[Section 6.4]{An}, 
\cite[Remark 2.3]{Ka}), and $\Delta_v(X,L_\l)\subset\R^4$ the Newton--Okounkov convex body corresponding to $X$, $L_\l$ and $v$ (see \cite{KK} for a definition).
We fix coordinates $(y_1,y_2,y_3,y_4)$ in $\R^4$ so that $v(x^{k_1}y^{k_2}z^{k_3}t^{k_4})=(k_1,k_2,k_3,k_4)$.
The valuation $v$ is natural from a geometric viewpoint: if $v(f)=(k_1,k_2,k_3,k_4)$ then $k_1$ is the order of vanishing of $f$ along the hypersurface $\{x=0\}$, while $k_2$ is the order of vanishing of $(x^{-k_1}f)|_{\{x=0\}}$ along the hypersurface $\{x=y=0\}\subset\{x=0\}$ and so on.

\begin{prop} \label{p.NO} Define a unimodular linear transformation of $\R^4$ by the formula
$$\varphi:(y_1,y_2,y_3,y_4)\mapsto(-y_1,-y_1-y_2,y_3+2y_4,y_4).$$
Then $\SP_\l=\varphi(\Delta_v(X,L_\l))+(\l_1,\l_1+\l_2,0,0).$
In particular, $\Delta_v(X,L_\l)$ can be described by inequalities:
$$0\le y_1, y_2, y_3, y_4; \quad  y_1 \le \l_1; \quad  2(y_1+y_2)+y_3+2y_4 \le 2(\l_1+\l_2); $$ $$y_1+y_2+y_3+2y_4 \le \l_1+2\l_2; \quad  y_4\le \l_2$$
\end{prop}
\begin{proof}
Note that $|\SP_\l\cap \Z^4|=\dim V_\l$ as polynomials in $\l$ by \cite[Theorem 3.6]{K}.
Comparing the highest degree homogeneous parts in $\l_1$ and $\l_2$ on both sides and using Hilbert's theorem we get
$${\rm volume}(\SP_\l)=\frac1{4!}\deg p_\l(Sp_4/B),$$
where $p_\l:Sp_4/B\to\P(V_\l)$ is the projective embedding of the flag variety corresponding to the weight $\l$.
Hence, to compare $\SP_\l$ and $\Delta:=\varphi(\Delta_v(X,L_\l))+(\l_1,\l_1+\l_2,0,0)$ it is enough to show that $\SP_\l\subset\Delta$.
Since both convex bodies have the same volumes the inclusion will imply the exact equality.

We now check that $\SP_\l\subset\Delta$.
There is a natural embedding $X\hookrightarrow\P^3\times IG(2,4)$; $(a,l)\in a\times l$, where
$IG(2,4)$ is the Grassmannian of isotropic planes in $\C^4$.
Let $p_{\om_1}$, $p_{\om_2}$ denote the projections of $X$ to the first and second factor, respectively.
Then $L_{\om_1}=p_{\om_1}^*\Oc_{\P^3}(1)$ and $L_{\om_2}=p_{\om_2}^*\pi^*\Oc_{\P^4}(1)$ where $\pi:IG(2,4)\to\P^4$ is the Pl\"ucker embedding.
Hence, $H^0(X,L_{\om_1})=\langle 1,-x,y+xz,z\rangle$ and $H^0(X,L_{\om_2})=\langle 1,-(y+2xz+x^2t), z+xt, yt-z^2,t\rangle$.
By taking the lowest order terms of basis sections we get that $\Delta_v(X,L_{\om_1})$ contains the simplex with the vertices
$$(0,0,0,0),(1,0,0,0),(0,1,0,0),(0,0,1,0)$$
and $\Delta_v(X,L_{\om_2})$ contains the simplex with the vertices
$$(0,0,0,0),  (0,1,0,0), (0,0,2,0), (0,0,0,1).$$
It is easy to check that $\varphi$ takes these two simplices to  $\SP_{\om_1}-(1,1,0,0)$ and $\SP_{\om_2}-(0,1,0,0)$, respectively.
Since $L_\l=L_{\om_1}^{\otimes\l_1}\otimes L_{\om_2}^{\otimes\l_2}$ the super-additivity of Newton--Okounkov bodies (see \cite[Theorem 4.9(3)]{KK}) implies that $\Delta_v(X,L_{\l})$ contains the Minkowski sum $\l_1\Delta_v(X,L_{\om_1})+\l_2\Delta_v(X,L_{\om_2})$.
Hence, $\Delta$ contains $\l_1 \SP_{\om_1}+\l_2 \SP_{\om_2} =\SP_\l$ as desired.
\end{proof}

\begin{example}
Take $\l=\rho$, i.e., $\l_1=\l_2=1$.
The projective embedding $p_\rho:Sp_4/B\to \P(V_\rho)$ comes from the composition of maps
$$Sp_4/B\hookrightarrow\P^3\times IG(2,4)\stackrel{\id\times\pi}{\longrightarrow} \P^3\times \P^4 \stackrel{\mbox{\tiny Segre}}{\longrightarrow}\P^{19}.$$
The image of $Sp_4/B$ is contained in $\P(V_\rho)\subset\P^{19}$.
In coordinates $(x,y,z,t)$, the embedding $Sp_4/B\hookrightarrow \P(V_\rho)\subset\P^{19}$ takes the point $(x,y,z,t)$ to
$$\begin{pmatrix}
1\\
-x\\
y+xz\\
z
\end{pmatrix}
\times
\begin{pmatrix}
1&&-(y+2xz+x^2t)&&z+xt&&yt-z^2&&t\\
\end{pmatrix}$$
Applying the valuation $v$ we get all
$16=\dim V_\rho$ integer points in $\SP_\rho$ (vertices of $\SP_\rho$ are underlined).
$$\underline{(0,0,0,0)}, (0,1,0,0), (0,0,1,0), (0,0,2,0), \underline{(0,0,0,1)},$$
$$\underline{(1,0,0,0)}, \underline{(1,1,0,0)}, (1,0,1,0), \underline{(1,0,2,0)}, \underline{(1,0,0,1)},$$
$$\hspace{1.9cm} \underline{(0,2,0,0)}, (0,1,1,0),  \underline{(0,1,2,0)}, \underline{(0,1,0,1)},$$
$$\hspace{1.9cm} \hspace{1.9cm} \hspace{1.9cm} \underline{(0,0,3,0)}, \underline{(0,0,1,1)}.$$
\end{example}
\begin{remark} Similar arguments can be applied to the valuation $v':=v^{\w_0'}$ corresponding to the decomposition $\w_0'=s_2s_1s_2s_1$, i.e., to the flag
of translated Schubert subvarieties
$$s_2s_1s_2s_1X_{id}\subset s_2s_1s_2X_{s_1}\subset s_2s_1X_{s_2s_1}\subset s_2X_{s_1s_2s_1}\subset X.$$
It is easy to check that the Newton--Okounkov body $\Delta_{v'}(X,L_\lambda)$ is obtained from $\Delta_{v}(X,L_\lambda)$ by the unimodular linear transformation $(y_1,y_2,y_3,y_4)\mapsto(y_4,y_3,y_2,y_1)$.
This agrees with the fact that symplectic DDO polytopes corresponding to $\w_0$ and $\w_0'$ are also the same up to an affine transformation (see \cite[Example 3.4]{K}).
\end{remark}
\begin{remark}
In \cite{Ka}, the Newton--Okounkov bodies of flag varieties for a different valuation $v_{\w_0}$ were identified with string
polytopes of \cite{L}.
Namely, $v_{\w_0}$ is the highest term valuation associated with the flag of Schubert subvarieties corresponding to the terminal subwords of $\w_0$.
For $Sp_4$ and $\w_0=s_1s_2s_1s_2$, this is the flag $X_{\id}\subset X_{s_2}\subset X_{s_1s_2}\subset X_{s_2s_1s_2}\subset X$.
By Remark \ref{r.DDO}, the polytopes $\Delta_{v_{\w_0}}(X,L_\l)$ and $\Delta_{v^{\w_0}}(X,L_\l)$
are not combinatorially equivalent (they have different number of vertices).
In particular, one can not expect a straightforward relation between valuations $v_{\w_0}$ an $v^{\w_0}$ (cf. \cite[Remark 2.3]{Ka}).
\end{remark}
\subsection{Newton--Okounkov polytopes of Schubert varieties}
We now identify (the unions of) faces of $\SP_\l$ obtained in Example \ref{e.Sp_4} with
generalized Newton--Okounkov polytopes of Schubert subvarieties of $X$.
This allows us to extend results of \cite{KST} on Schubert calculus in terms of polytope rings from Gelfand--Zetlin polytopes and $GL_n$ to the symplectic DDO polytope $\SP_\l$ and $Sp_4$.
A different extension was previously obtained in \cite{I} for the string polytopes of $Sp_4$ associated with $\w_0=s_2s_1s_2s_1$ (this polytope coincides up to a unimodular change of coordinates with the {\em symplectic Gelfand--Zetlon polytope} \cite[Corollary 6.2]{L}).

We say that the union of faces $\Delta_w=\bigcup_{F\subset \SP_\l} F$ is a {\em generalized Newton--Okounkov polytope} of a 
Schubert subvariety $X_w$ if $|\Delta_w\cap\Z^4|=\dim H^0(X_w,L_\l|_{X_w})$ as polynomials in $\l$.
In particular, $\SP_\l=\Delta_{w_0}$ and any vertex of $\SP_\l$ is a valid choice for $\Delta_{id}$.
Corollary \ref{t.Schubert} immediately yields the following choices for the other Schubert varieties:
$$\Delta_{s_1}=H_2^+\cap H_3^+\cap H_4; \quad \Delta_{s_2s_1}=H_3^+\cap H_4^+; \quad \Delta_{s_1s_2s_1}=H_4^+;$$
$$\Delta_{s_2}=H_1^+\cap H_3^+\cap H_4^+; \quad \Delta_{s_1s_2}=(H_1^+\cap H_4)^+\cup (H_2^+\cap H_4^+);
\quad \Delta_{s_2s_1s_2}=H_1^+\cup H_2^+\cup H_3^+,$$
where $H_1^+$,\ldots, $H_4^+$ denote the facets of $\SP_\l$ given by equations $y_1=0$, $2y_2=y_3$, $y_3=2y_4$, $y_4=0$, respectively.
Applying results of \cite[Section 2]{KST} and \cite[Theorem 4.1]{Ka2} to $\SP_\l$ we can multiply Schubert cycles in $H^*(X,\Z)$ by intersecting their generalized Newton--Okounkov polytopes if the latter are transverse.
For instance,
$$[X_{s_1s_2s_1}]\cdot [X_{s_2s_1s_2}]=[\Delta_{s_1s_2s_1}\cap\Delta_{s_2s_1s_2}]=[\Delta_{s_1s_2}\cup\Delta_{s_2s_1}]=
[X_{s_1s_2}]+[X_{s_2s_1}].$$
Using techniques of \cite[Section 2]{KST} we can realize the Schubert calculus on $X$ in terms of $\SP_\l$.
Namely, \cite[Formula (1)]{KST} gives four linear relations between (equivalence classes of) facets of $\SP_\l$:
$$[H_1^+] + [H_2^-]=[H_1^-]; \quad 2[H_2^+] + [H_3^-]=[H_2^-];$$
$$\ [H_2^+]+[H_3^-]=[H_3^+]; \quad 2[H_3^+]+[H_4^-]=[H_4^+],$$
where $H_1^+$,\ldots, $H_4^+$ denote the facets of $\SP_\l$ given by equations $y_1=\l_1$, $y_2=y_1+\l_2$, $y_3=y_2+\l_2$, $y_4=\l_2$, respectively.
Using these relations we can get new generalized Newton--Okounkov polytopes, e.g.
$$
\Delta'_{s_1s_2s_1}=H_2^-\cup H_3^-\cup H_4^-;\quad \Delta'_{s_2s_1s_2}=H_1^-,
$$
such that the intersections $\Delta_v\cap\Delta'_w$ are transverse for all $v$ and $w$.

\section{Combinatorics of geometric mitosis and open questions}\label{s.comb}
We now discuss combinatorics of mitosis on admissible balanced parapolytopes.
We outline a combinatorial algorithm for generating faces that appear in Corollary \ref{t.Schubert}.
For Gelfand--Zetlin polytopes, this algorithm reduces to mitosis of \cite{KnM} on pipe dreams.
Generalizing combinatorics of Example \ref{e.Sp_4} we define mitosis on {\em skew pipe dreams} for $Sp_{2n}$.
In the end of this section, we formulate open questions.

\subsection{Mitosis on vertex cone}
Let $P\subset\R^d$ be an admissible $\lb$-balanced parapolytope with the lowest vertex $0$.
Since the faces that appear in Corollary \ref{t.Schubert} are obtained from the vertex $0\in P$ by mitosis operations they contain $0$.
Hence, to describe these faces it is enough to consider the combinatorics of  the  vertex cone $C_0$ of $P$ at $0$ and not the whole $P$.
Recall that the {\em vertex cone} $C_a$ of a vertex $a\in P$ by definition consists of all $b\in\R^d$ such that $a+\l(b-a)\in P$ for some $\l\ge 0$.
Let $H_1$,\ldots, $H_{d'}$  be the facets of $C_0$.
Note that $d'\ge d$, and $0$ is a simple vertex of $P$ if and only if $d=d'$.
Facets $H_j$ correspond to homogeneous inequalities $l_j\ge0$ that define $C_0$.

Fix $i\in\{1,\ldots, r\}$ and consider $c\in\R^d/\R^{d^i}$.
Since $P$ is a parapolytope we have that $\Pi_c:=P\cap(c+\R^{d_i})$ is given by inequalities
$\mu^i_j(c)\le x^i_j\le\nu^i_j(c)$ for $j=1$,\ldots, $i_d$,
where $\mu^i_j(c)$ are linear functions.
If $P$ is admissible then the parallelepiped $\Pi_0$ is a segment (or a point if $(\l,\a_1)=0$) given by inequality 
$0\le x^i_1\le(\l,\a_1)$ and equalities $x_j^i=0$ for $j=2,$\ldots, $d_i$.
So $\mu^i_1(0)=\mu^i_j(0)=\nu^i_j(0)=0$ for all $j\ge 2$, and functions $\mu^i_j(c)$ and $\nu^i_j(c)$ are all homogeneous except for possibly $\nu^i_1(c)$.
In particular, $U_c:=C_0\cap (c+\R^{d_i})$ is given by inequalities $\mu^i_j(c)\le x^i_j\le \nu^i_j(c)$ for $j=2$,\ldots, $i_d$ and $\mu^i_1\le x^i_j$, that is, $U_c$ is almost a parallelepiped: it might be not bounded only in $x_1^i$-direction (if $\nu^i_1(c)\ne\mu^i_1(c)$).
Note that mitosis on parallelepipeds defined in Section \ref{s.para} never produces faces that lie in the facet $x_1=\nu_1$ (unless $\mu_1=\nu_1$).
Hence, the definition of mitosis on parallelepipeds goes verbatim for the faces of $U_c$.

Let $\Gamma\subset C_0$ be a face.
Choose $c\in \G^\circ$.
Choose facets $H_{j_1}$,\ldots, $H_{j_\ell}$ of $C_0$ such that
every face of $U_c$ can be uniquely represented as the intersection of $U_c$ with some of these facets.
In particular, $\Gamma_c=H_{i_1}\cap\ldots\cap H_{i_k}\cap U_c$ for some $\{i_1,\ldots,i_k\}\subset\{j_1,\ldots,j_\ell\}$, hence,
$\Gamma=H_{i_1}\cap\ldots\cap H_{i_k}\cap P$.
Then mitosis on $U_c$ tells us which facets in $\Gamma=H_{i_1}\cap\ldots\cap H_{i_k}$ should be deleted and which facets  added in order to get all faces in $M_i(\G)$.
We get a purely combinatorial operation $M_i$ on the subsets of the set $\{H_{i_1},\ldots,H_{i_\ell}\}$.
Facets of $C_0$ and all operations $M_i$ can be encoded by diagrams similar to pipe dreams.
Usual pipe dreams correspond to the case when $C_0$ is a vertex cone of the Gelfand--Zetlin polytope, or equivalently, $C_0$ is the cone of adapted strings in type $A$
(see \cite[Theorem 5.1]{L}).

Below we consider a new combinatorial algorithm that arises from the geometric mitosis on the cone of adapted strings in type $C$.

\subsection{Mitosis on skew pipe dreams}
Let $G=Sp_{2n}$, i.e., $r=n$ and $d=n^2$.
Take the reduced decomposition $\w_0=(s_n s_{n-1}\ldots s_2 s_1 s_2\ldots s_{n-1}s_n)\ldots (s_2s_1s_2)(s_1)$.
Then $\R^{n^2}=\R^{n}\oplus\R^{2n-2}\oplus\R^{2n-4}\oplus\ldots\oplus\R^2$.
Note that from now on $s_1$ corresponds to the longer root in accordance with \cite{L}.
Let $P\subset\R^{n^2}$ be a parapolytope with the lowest vertex $0$ such that
the vertex cone $C_0$ is defined by inequalities
$$0\le x^i_2\le x^{i-1}_4\le x^{i-2}_6\le\ldots\le x^2_{2i-2}\le x^1_{i}\le x^2_{2i-3}\le\ldots\le x^{i-2}_5\le x^{i-1}_3\le x^i_1 \eqno(*)$$
for all $i=1$,\ldots, $n$.
There are $n^2$ inequalities in $(*)$, in particular, $0$ is a simple vertex of $C_0$.
The cone $C_0$ is exactly the cone of adapted strings for the decomposition $\w_0$  (see \cite[Theorem 6.1]{L}).

Faces of $C_0$ can be encoded by {\em skew pipe dreams}.
A {\em skew pipe dream} of size $n$ is a $(2n-1)\times{n}$ table whose cells are either empty or filled with $+$.
Only cells $(i,j)$ with $n-j< i < n+j$ are allowed to have $+$.
When drawing a skew pipe dream we omit cells $(i,j)$ that do not satisfy these inequalities.
For instance, all tables of Example \ref{e.Sp_4} are skew pipe dreams of size $n=2$.
There is a bijective correspondence between faces of $C_0$ and skew pipe dreams: to get the skew pipe dream  $D(\G)$ corresponding to a face $\G\subset C_0$ replace an inequality $x^{k'}_{l'}\le x^k_l$ (or $0\le x^k_l$) in $(*)$ by $+$ at cell
$$
\left\{
\begin{array}{lll}(n+k-1,k+\frac{l-1}{2}) & \mbox{ if }& l \mbox{ is odd}, \ k\ne 1\\
(n-k+1,k+\frac{l}2-1) & \mbox{ if }& l \mbox{ is even},  \ k\ne 1\\
(n,l)& \mbox{ if } &k=1\\
\end{array}
\right. \eqno(**)
$$
whenever $x^{k'}_{l'}=x^k_l$ (or $0=x^k_l$) identically on $\G$.
Table $(**)$ gives a bijection between coordinates $x^k_l$ and (fillable) cells of a skew pipe dream.

\begin{example} Let $n=3$.
The bijection between cells and coordinates given by (**) is depicted on the left.
The skew pipe dream $D(G)$ of the face $\G=\{0=x^1_1; \ 0=x^2_2=x^1_2; \ 0=x^3_2; \ x^2_3=x^3_1 \}$ is depicted on the right.

$$
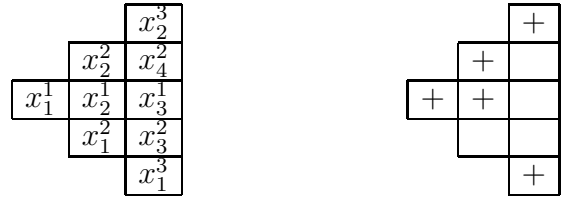
\begin{array}{|c|}
\hline
x_1^1 \\
\hline
\end{array}
\begin{array}{|c|}
\hline
x_2^2\\
\hline
x_2^1\\
\hline
x_1^2\\
\hline
\end{array}
\begin{array}{|c|}
\hline
x_2^3\\
\hline
x_4^2\\
\hline
x_3^1\\
\hline
x_3^2\\
\hline
x_1^3\\
\hline
\end{array}
\hspace{3cm}
\begin{array}{|c|}
\hline
+ \\
\hline
\end{array}
\begin{array}{|c|}
\hline
+\\
\hline
+\\
\hline
\\
\hline
\end{array}
\begin{array}{|c|}
\hline
+\\
\hline
\\
\hline
\\
\hline
\\
\hline
+\\
\hline
\end{array}$$
\end{example}

The bijection between faces of $C_0$ and skew pipe dreams transforms geometric mitosis on faces of $C_0$ into the following combinatorial rule.
We use terminology of \cite[Section 3]{M}.
Given a skew pipe dream $D$ of size $n$, define
$$\start_i(D) = \min\{S_{n-i+1},S_{n+i-1}+1\},$$
where $S_j$ denotes the column index of the leftmost empty cell in row $j$, i.e.,
$$\start_i(D)= \min\{\min({j \ | \ (n-i+1, j) \notin D}), \min({j \ | \ (n+i-1, j) \notin D})+1\},$$
so the $(n\pm (i-1))$-th rows of $D$ are filled solidly with crosses in the region to the right and upward of cell $(\start_i(D)-1,n+i-1)$.
Let
$$\cJ^-(D) = \{\mbox{columns } j \mbox { strictly to the right of }
\start_i(D) \ |\   (n-i+2, j)  \mbox{ has no cross in } D\}.$$
and
$$\cJ^+(D) = \{\mbox{columns } j \mbox { strictly to the right of }
\start_i(D) \ |\   (n+i,j) \mbox{ has no cross in } D\}.$$
For $p\in\cJ^\pm(D)$, we now construct the {\em offspring} $D_p^\pm$ in two or three steps as follows.
\begin{enumerate}
\item If $p\in\cJ^-(D)$, to construct $D_p^-$ delete the cross at $(n-i+1, p)$ from $D$.
If $p\in \cJ^+(D)$, to construct $D_p^+$ delete the cross at $(n+i-1, p)$.

\item Take all crosses in row $n-i+1$ of $\cJ^-(D)$ and in row $n+i-1$ of $\cJ^+(D)$ that are to the right of column $p$, and move each one down to the empty box below it in row $n-i+2$ and in row $n+i$, respectively.

\item If $p\notin \cJ^-(D)\cap\cJ^+(D)$ or $i=1$, then we are done with both $D^-_p$ and $D^+_p$.
Otherwise, an additional step is required to construct $D^+_p$: move the cross at $(n-i+1,p)$
to the empty box below it in row $n-i+2$.
\end{enumerate}

\begin{defin}
The $i$-th mitosis operator sends a skew pipe dream $D$ to
$${\rm mitosis}_i(D) = \{D^-_p \ | \ p\in\cJ^-(D)\}\cup \{D^+_p \ | \ p\in\cJ^+(D)\}.$$
\end{defin}
Note that the $i$-th mitosis affects only rows $n\pm(i-1)$, $n-i+2$ and $n+i$, and ${\rm mitosis_i}(D)$ is empty if both $\cJ^+$ and $\cJ^-$ are empty.
It is easy to  check that under the above bijection between faces of $C_0$ and skew pipe dreams
we have
$${\rm mitosis}_i(D(\G))=M_i(\G).$$
In particular, for $n=2$ this combinatorial algorithm yields exactly the same tables as in Example \ref{e.Sp_4}.
\begin{example} Let $n=3$ and $i=2$.
$$D=\begin{array}{|c|}
\hline
\ \ \\
\hline
\end{array}
\begin{array}{|c|}
\hline
+\\
\hline
\\
\hline
+\\
\hline
\end{array}
\begin{array}{|c|}
\hline
+\\
\hline
+\\
\hline
\\
\hline
+\\
\hline
+\\
\hline
\end{array}
\stackrel{{\rm mitosis_2}}{\longrightarrow}
\left\{
D_2^+=\begin{array}{|c|}
\hline
\ \ \\
\hline
\end{array}
\begin{array}{|c|}
\hline
\\
\hline
+\\
\hline
\\
\hline
\end{array}
\begin{array}{|c|}
\hline
+\\
\hline
\\
\hline
+\\
\hline
+\\
\hline
+\\
\hline
\end{array} \ , \
D_2^-=\begin{array}{|c|}
\hline
\ \ \\
\hline
\end{array}
\begin{array}{|c|}
\hline
\\
\hline
\\
\hline
+\\
\hline
\end{array}
\begin{array}{|c|}
\hline
+\\
\hline
\\
\hline
+\\
\hline
+\\
\hline
+\\
\hline
\end{array} \ , \
D_3^-=\begin{array}{|c|}
\hline
\ \ \\
\hline
\end{array}
\begin{array}{|c|}
\hline
+\\
\hline
\\
\hline
+\\
\hline
\end{array}
\begin{array}{|c|}
\hline
+\\
\hline
\\
\hline
\\
\hline
+\\
\hline
+\\
\hline
\end{array}
\right\}
$$
In this example,  $\start_i(D)=1$, $\cJ^-(D)=$\{columns 2, 3\} and $\cJ^+(D)=$\{column 2\}.
\end{example}
\subsection{Open questions}
It is tempting to use combinatorial mitosis on skew pipe dreams to produce an explicit realization of generalized Newton--Okounkov polytopes for Schubert varieties on $Sp_{2n}/B$ by collections of faces of symplectic string polytopes.
While such a realization exists by general properties of string polytopes (see \cite[Section 5.5]{Mi} for more details) an explicit description is known only for $n=2$ (see \cite{I}).
However, the symplectic string polytopes associated with $\w_0=(s_n s_{n-1}\ldots s_2 s_1 s_2\ldots s_{n-1}s_n)\ldots (s_2s_1s_2)(s_1)$ are not parapolytopes with respect to decomposition
$\R^{n^2}=\R^{n}\oplus\R^{2n-2}\oplus\R^{2n-4}\oplus\ldots\oplus\R^2$ (already for $n=2$), so Corollary \ref{t.Schubert} can not be directly applied to them.

As we have seen in Section \ref{s.Sp_4}, the symplectic DDO polytope in the case of $Sp_4$ turned out to be a more suitable candidate for constructing explicit generalized Newton--Okounkov polytopes using Corollary \ref{t.Schubert}.
Symplectic DDO polytopes can also be constructed for $Sp_{2n}$ using reduced decomposition
$\w_0'=(s_n\ldots s_1)^n$ rather than $\w_0$
(note that for $n=2$ we have $\w_0=\w_0'$).
In an ongoing project with M. Padalko, we aim to describe these polytopes explicitly by inequalities, study combinatorics of their geometric mitosis and applications to the Schubert calculus on $Sp_{2n}$.

It is also interesting to check whether the Newton--Okounkov polytopes of flag varieties associated with the lowest term valuation 
$v^{\w_0}$ (see Section \ref{s.Sp_4}) are good candidates for applying geometric mitosis to the Schubert calculus.
Proposition \ref{p.NO} suggests that this might be the case.
Recall that theory of Newton--Okounkov polytopes can be used to construct toric degenerations
\cite{An}.
If a Newton--Okounkov polytope $P$ of the flag variety $X$ satisfies conditions of Corollary \ref{t.Schubert} and $X_P$ is the toric degeneration of $X$ associated with $P$
then it is natural to expect that collections of faces given by geometric mitosis yield  degenerations of Schubert varieties to (reduced) toric subvarieties of $X_P$.

\end{document}